\documentclass[12pt,psfig,reqno]{amsart}
\usepackage{txfonts}
\usepackage{amscd}
\usepackage{cite}

\usepackage{epsfig}

\usepackage{mathdots}
\usepackage{amssymb}
\usepackage{amsmath}
\usepackage{amsthm}
\usepackage{amsfonts}

\usepackage{verbatim}
 \usepackage{graphicx}
 \usepackage{epstopdf}
 \usepackage[usenames]{color} 
 \usepackage{cancel}
 \numberwithin{equation}{section}
 
\newtheorem{definition}{Definition}[section]
\newtheorem{theorem}[definition]{Theorem}
\newtheorem{conjecture}[definition]{Conjecture}

\newtheorem{claim}[definition]{Claim}

\newtheorem{lemma}[definition]{Lemma }

\newtheorem{remark}[definition]{Remark}

\numberwithin{equation}{section} 

\begin{document}
\baselineskip=17pt
\title{The Beurling density of the spectrum of self-similar measure generated by Hadamard triple}
\author{Zong-Sheng Liu$^1$}
\author{Xiao-Yu Yan$^{2\;*}$}

\address{$^1$College of Mathematics and Statistics , Hengyang Normal University, Hengyang, Hunan 421002, China; and
$^2$Key Laboratory
of Computing and Stochastic Mathematics (Ministry of Education), School of Mathematics and Statistics, Hunan Normal University, Changsha,
Hunan 410081, P.R. China.}
\email{lzsheng1992@163.com\;(Z.-S. Liu)}
\email{xyyan1103@163.com\;(X.-Y. Yan)}

\thanks{This work is supported in part by the NNSF of China (Nos.12401107, 12071125, 12371072) and the Hunan Provincial NSF (No.2026JJ60112) and by the Application-Oriented Characterized Disciplines, Double First-Class University Project of Hunan Province (Xiang jiao tong [2022]337).}

\begin{abstract}
Let $D\subset\Bbb Z$ with cardinality $q\ge 2$, and let $b\in \Bbb Z$ with $q<b$, and let $\mu:=\mu_{b,D}$ be the associated self-similar measure. It is well known that if there exists $L\subset \Bbb Z$ such that $(b,D,L)$ be a Hadamard triple, then the Beurling dimension of the spectrum of $\mu$ exhibits an intermediate structural property. In this paper, we establish a stronger result that both Beurling dimension and Beurling density of the spectra of $\mu$ can achieve full flexibility simultaneously. More precisely, for any $t\in(0, \frac{\log q}{\log b})$ and $s\in [0,\infty]$, there exists a spectrum $\Lambda:=\Lambda_{t,s}$ of $\mu$ such that
$$\dim_{Be}(\Lambda)=t,\quad D_t^+ (\Lambda)=s.$$
Here,  $\dim_{Be}$ and $D_t^+$ denote the Beurling dimension and the $t$-Beurling density, respectively. We further prove that the set of such spectrum whose Beurling dimension and Beurling density are equal to any fixed $t$ and $s$ has the cardinality of the continuum.
\par This work generalizes a previous result of Lu \cite{Lu}, answers an open question raised by Dai, Fu and He \cite[Conjecture 5.3]{DaiFuHe}, and sheds new light on the fine structural properties of spectra for singularly continuous spectral measures.
\end{abstract}
\maketitle


\section{\bf Introduction\label{sect.1}}
\setcounter{equation}{0}
\par Let $\mu$ be a Borel probability measure with compact support on $\Bbb R^n$. $\mu$ is called a spectral measure if there exists a countable set $\Lambda \subset \Bbb R^n$ such that $E(\Lambda):=\{e^{2\pi i\langle \lambda,x\rangle }: \lambda \in \Lambda\}$ forms an orthonormal basis for $L^2(\mu)$, where $\langle\cdot, \cdot \rangle$ is the standard inner product in $\Bbb R^n$. The set $\Lambda$ is then called a spectrum for $\mu$. If a spectral measure $\mu$ is the Lebesgue measure on a measurable set $\Omega$, then we say that $\Omega$ is a spectral set. The theory of spectral measures could date back to 1974. In the seminal paper\cite{F}, Fuglede conjectured that $\Omega$ is a spectral set if and only if $\Omega$ is a translational tile. The conjecture remained open until 2004 and it was disproved by Tao\cite{T}and Matolcsi\cite{M}. Although the conjecture was proven to be false in both directions on $\Bbb R^n$ for $n\geq 3$\cite{KM}, but it is still open for $n=1, 2$ and the study of the conjecture remains active until today\cite{FFLS}, \cite{LM}. Moreover, it has led to the development of the research of spectral measures. It is very attractive when we consider the spectrality of fractal measures. In 1998, Jorgensen and Pedersen\cite{JP} initiated an investigation of spectral property of fractal measures. They discovered the first continuous singular spectral measure: the standard middle-fourth Cantor measure. This celebrated discovery opened up a new field in the orthogonal harmonic analysis of fractal measures. Since then, there has been a lot of research on self-similar/self-affine spectral measures (see \cite{AL1},\cite{D}-\cite{DC},\cite{DutHauLai},\cite{HLL},\cite{HL},\cite{LW}, \cite{Li}-\cite{LMW}, \cite{LL}, \cite{LDL},\cite{W} and so on), as well as the convergence properties of the associated Fourier series\cite{DutHanSun2},\cite{S2} and the peculiar properties of spectra for fractal measures\cite{ADH},\cite{AL2},\cite{FHW}.

Let $M\in M_n(\Bbb R)$ be an expanding matrix (all its eigenvalues have modulus strictly greater than 1) and $D$ be a set in $\Bbb R^n$ with cardinality $1<\# D<\infty$. This pair $(M, D)$ defines the affine iterated function system (IFS)
$$
\Phi=\Big\{\phi_d(x):\phi_d(x)=M^{-1}(x+d),d\in D \Big\}.
$$
Hutchinson\cite{H} proved that this IFS determines a unique nonempty compact set $T$, called an attractor, and a Borel probability measure $\mu_{M,D}$ supported on $T$ satisfying
\[
T=\bigcup_{d\in D} \phi_d(T), \quad \mu_{M,D}(\cdot)=\frac{1}{\# D}\sum_{d\in D} \mu_{M,D}(\phi_d^{-1}(\cdot)).
\]
The Borel probability measure $\mu_{M,D}$ is called a self-affine measure and $T$ a self-affine set. In particular, if $M=\rho U$, where $\rho> 1$ and $U$ is an orthonormal matrix, then the corresponding measure  $\mu_{M,D}$ is a self-similar measure and $T$ is called a self-similar set.  The self-affine measure $\mu_{M,D}$ can be expressed as an infinite convolutions of finite atomic measures
\begin{align*}
\mu_{M,D}=\delta_{M^{-1}D}*\delta_{M^{-2}D}*\cdots*\delta_{M^{-k}D}*\cdots,
\end{align*}
where $\delta_E=\frac{1}{\# E}\sum_{e\in E}\delta_e$, $E$ is a finite set and $\delta_e$ is the Dirac measure at the point $e\in E$. The equality holds in weak sense.

It is difficult to classify all spectral measures completely. He et al.\cite{HLL} proved that a spectral measure must be pure type, i.e., it is either discrete with finite support, or absolutely continuous or singularly continuous with respect to the Lebesgue measure. In their research, the Beurling density plays a key role. Beurling density was firstly introduced by Landau\cite{Landau} to study the frame properties of complex exponentials. Dutkay et al. \cite{DHSW} proposed to use the Beurling dimension as some general criteria for the existence of Fourier frame. Let $\Lambda$ be a countable set in $\Bbb R^n$. The $r$-(upper) Beurling density of $\Lambda$ is defined by
\begin{align*}
D_r^+(\Lambda)=\limsup_{h\to\infty}\sup_{x\in \Bbb R^n}\frac{\#(\Lambda \cap B(x,h))}{h^r},
\end{align*}
where $B(x, h)$ denotes the open ball centered at $x$ with radius $h$. Then we can define the Beurling dimension of $\Lambda$ by
\begin{align*}
\dim_{Be} \Lambda=\inf\{r: D_r^+(\Lambda)=0\}=\sup\{r: D_r^+(\Lambda)=\infty\}.
\end{align*}

On the other hand, it is interesting to study the spectra for a given spectral measure. Beurling dimension is very important for  characterizing  the size of spectra.  Dutkay et al.\cite{DHSW} and He et al.\cite{HKTW} proved that Beurling dimension of the spectrum for self-similar spectral measure which satisfies the open set condition is not greater than the Hausdorff dimension of its support. Moreover, they also proved that the two of dimensions are equal under certain condition. Dai et al.\cite{DaiHeLai} constructed a spectrum $\Lambda$ for a self-similar measure with consecutive digits such that $\dim \Lambda=0$. An and Lai\cite{AL2} further proved that there is spectrum with zero Beurling dimension for general self affine spectral measure. Wang and Zhang\cite{WZ} showed that the Beurling dimension of spectra for a class of Moran measures between $0$ and the Hausdorff dimension of its support. Tang and Wu\cite{TW} proved that under some mild condition similar to the self-similar case, the Beurling dimension of any spectrum of the self affine spectral measure is bounded by the pseudo Hausdorff dimension of its support. Shi \cite{Shi} proved that the Beurling dimension of its spectra is bounded by its upper entropy dimension. Iosevich et al.\cite{I} proved that the Beurling dimensions of its frame spectra are bounded by its the Fourier dimension and upper entropy dimension. That is to say, if $\mu$ is a spectral measure with spectrum $\Lambda$, then
\[
\dim_F\mu \leq \dim_{Be} \Lambda \leq \overline{\dim}_e \mu,
\]
where $\dim_F$ and $\overline{\dim}_e \mu$ denote the Fourier dimension and upper entropy dimension of $\mu$ respectively. Li et al.\cite{LWX} generalized this conclusion to extended frame spectral measure.

Along the above researches, there is  a natural conjecture concerning the intermediate value property of Beurling dimension of spectra:
\begin{conjecture}\label{conj1.1}
Let $\mu$ be a singular continuous spectral measure with compact support on  $\Bbb R^n$. Then for any $t\in [\dim_F\mu, \overline{\dim}_e \mu]$, there exists a spectrum $\Lambda_t$ of $\mu$ such that $\dim_{Be} \Lambda_t =t$.
\end{conjecture}
Li and Wu\cite{LW1,LW2} proved that Conjecture \ref{conj1.1} holds for several classes of measures. While Conjecture \ref{conj1.1} addresses intermediate value property for the Beurling dimension of spectra, a further deeper open problem is the so called {\it density intermediate value conjecture} (Conjecture 5.3 in \cite{DaiFuHe}), which concerns the intermediate value property of the upper $r$-Beurling density when we fix the Beurling dimension of spectra.
\begin{conjecture}\label{conj1.2}
Suppose that a singular continuous spectral measure admits at least one spectrum with Beurling dimension $r$. Let
\[
G=\left\{D_r^{+}(\Lambda): \Lambda \text{ is a spectrum of } \mu \text{ and } \dim_{Be}\Lambda=r\right\}
\]
be the set of all upper $r$-Beurling densities realized by spectra with fixed Beurling dimension $r$. Then $G$ is a bounded interval. Furthermore, for every $t\in G$, the collection
\[
\left\{\Lambda: 0\in \Lambda,\ D_r^{+}(\Lambda)=t,\ \dim_{Be}\Lambda=r\right\}
\]
has the cardinality of the continuum.
\end{conjecture}

Conjecture \ref{conj1.2} poses two core questions: for fixed Beurling dimension $r$, (i) the set of attainable upper $r$-Beurling densities forms an interval (intermediate value property); (ii) for each attainable density value, there exist continuum many spectra realizing this exact density. Up to now, very few measure models can verify Conjecture \ref{conj1.2}.

In this paper, we study the Beurling dimension and Beurling density of the spectra of $\mu:=\mu_{b,D}$ which satisfies
\begin{align}\label{1.1}
\mu_{b,D}(\cdot)=\frac{1}{q}\sum_{j=0}^{q-1} \mu_{b,D}(b(\cdot)-d_j).
\end{align}
where $D=\{0=d_0,d_1,\cdots,d_{q-1}\}$ and  $q<b$. To guarantee the spectral property of $\mu$, we require that there is $L\subset \{0,1,\cdots,b-1\}\pmod b$ such that $(b,D,L)$ forms a Hadamard triple. Without loss of generality, we suppose that $L\subset \{0,1,\cdots,b-1\}$. Our main results give an affirmative confirmation of Conjecture \ref{conj1.2} for this class of self-similar spectral measures. We obtain following results.

\begin{theorem}\label{th1.3}
Let $\mu$ be given by \eqref{1.1}. Then for any $t\in(0, \frac{\log q}{\log b})$, $s\in [0,\infty]$, there exists a spectrum $\Lambda:=\Lambda_{t,s}$ of $\mu$ such that $\dim_{Be}(\Lambda)=t, D_t^+ (\Lambda)=s$.
\end{theorem}

Furthermore, we can describe the cardinality of the spectra whose Beurling dimension and Beurling density are equal to any fixed $t$ and $s$. More precisely, we have the following surprising result.
\begin{theorem}\label{th1.4}
Let $\mu$ be given by \eqref{1.1}. Then for any $t\in(0, \frac{\log q}{\log b})$, $s\in [0,\infty]$, the set
\[
\mathcal{V}^{(s)}_{t}(\mu)=\{\Lambda: \Lambda \;\text{is a spectrum of}\;\mu \;\text{and}\;\dim_{Be}(\Lambda)=t,D_t^+ (\Lambda)=s\}
\]
has the cardinality of the continuum.
\end{theorem}

The above two theorems settle the density intermediate value conjecture (Conjecture \ref{conj1.2}) for the measure $\mu_{b,D}$. Theorem \ref{th1.3} shows that for any target Beurling dimension $t\in(0,\frac{\log q}{\log b})$, every $s\in[0,\infty]$ can be realized as the upper $t$-Beurling density of some spectrum with exact Beurling dimension $t$. Theorem \ref{th1.4} further demonstrates that for each pair $(t,s)$, there exist continuum many spectra achieving these prescribed dimension and density simultaneously. These results reveal the rich structure and flexibility of spectra for singular continuous spectral measures and the method we used can be generalized to general measures.

The paper is organized as follows: In Section 2, we introduce some notations, preliminary results and basic properties which will be used throughout the paper. In Section 3, we will prove Theorem \ref{th1.3}-\ref{th1.4}.

\section{\bf Preliminaries \label{sect.2}}

For any Borel probability measure $\mu$ on $\Bbb R$, its Fourier transform is defined by
$$
\hat{\mu}(\xi)=\int e^{-2\pi i\xi x} d\mu(x),\quad \xi\in \Bbb R.
$$
A set $\Lambda$ is an orthogonal set (a maximal orthogonal set) if $E(\Lambda)$ is an orthogonal family (a maximal orthogonal family) in $L^2(\mu)$. Let $f(x)$ be a complex function, the zero set of $f(x)$ is denoted by $\mathcal{Z}(f)$, i.e., $\mathcal{Z(}f)=\{x: f(x) = 0\}$. It is well known that a countable set $\Lambda\subset \Bbb R$ is an orthogonal set of a measure $\mu$ if and only if $(\Lambda-\Lambda)\setminus \{\mathbf{0}\}\subset \mathcal{Z}(\hat{\mu})$.
For the self-similar measure $\mu$ given by \eqref{1.1}, we can write it as
\[
\mu = \delta_{b^{-1}D} * \delta_{b^{-2}D} * \cdots * \delta_{b^{-n}D} * \cdots
= \mu_n * \mu_{>n},
\]
where $\mu_n := \delta_{b^{-1}D} * \delta_{b^{-2}D} * \cdots * \delta_{b^{-n}D}$. The following lemma gives a test to decide whether a countable set is a spectrum of $\mu$ or not.

\begin{lemma}\label{lem2.1}\cite{AHH}
With above notations, suppose that $\{\alpha_n\}_{n=1}^{\infty}$ is an increasing sequence of integers and $\delta$ is a small positive number.
If $0\in \Lambda_{\alpha_n}\subset \Lambda_{\alpha_{n+1}}$, $\Lambda_{\alpha_n}$ is a spectrum of $\mu_{\alpha_n}$ and
\[
\inf_{\lambda\in \Lambda_{\alpha_n},\, \xi\in [0,\delta]} \big|\widehat{\mu}_{>\alpha_n}(\lambda+\xi)\big|^2 >0
\]
for any $n\ge 1$, then $\Lambda=\bigcup_{n=1}^{\infty}\Lambda_{\alpha_n}$ is a spectrum of $\mu$.
\end{lemma}

The following lemma is important in constructing the spectrum for our main result.

\begin{lemma}\cite{Lu}\label{lem2.2}
Let $(b,D,L)$ form a Hadamard triple on $\mathbb R$ with $L\subset \{0,1,\cdots,b-1\}$, and let $\mu$ be given as \eqref{1.1}. There exist an integer $N_1>0$ and a constant $\epsilon_1>0$  such that for any $\xi\in [0,1]$, there exist two integer
\[
\ell_1^{(\xi)},\ell_2^{(\xi)}\in \{0,1,\cdots,b^{N_1}-1\}
\]
such that $\ell_1^{(\xi)}<\ell_2^{(\xi)}$ and
\[
|\widehat{\mu}(\xi+\ell_i^{(\xi)})|\ge \epsilon_1
\]
for $i=1,2$.
\end{lemma}

To estimate the Beurling dimensions of the spectra, we need following lemmas.

\begin{lemma}\label{lem2.3}\cite{CKS}
Let $\Lambda\subset\mathbb{R}$ be a countable set. Then
\[
\dim_{Be}(\Lambda)
= \varlimsup_{h\to\infty}\sup_{x\in\mathbb{R}}
\frac{\log \#\big(\Lambda \cap \big(x-h,\,x+h\big)\big)}{\log h}.
\]
\end{lemma}

\begin{lemma}\label{lem2.4}
Let $\Lambda_1$ and $\Lambda_2$ be a countable set in $\mathbb{R}$. Then\\
\noindent(i)\;$\dim_{Be}(\Lambda_1\cup\Lambda_2)=\max\big\{\dim_{Be}(\Lambda_1),\dim_{Be}(\Lambda_2)\big\}$;\\
\noindent(ii)\;If $r=\dim_{Be}(\Lambda_1)>\dim_{Be}(\Lambda_2)$, then $D_r^{+}(\Lambda_1\cup\Lambda_2)=D_r^{+}(\Lambda_1)$.
\end{lemma}

\begin{lemma}\label{lem2.5}\cite{AL2}
Let $b>1$. If $\Lambda\subset \Bbb R$ is a $b$-lacunary set ( $\lambda \ge b\lambda'$ for any $\lambda,\lambda'\in \Lambda$ with $\lambda>\lambda'$ ), then $\dim_{Be}(\Lambda)=0$.
\end{lemma}

\begin{lemma}\label{lem2.6}
(i)For any $t\in(0, \frac{\log q}{\log b})$, $s\in (0,\infty)$, if $\frac{t\log b}{\log q}\in \Bbb Q^c$, there exist sequences $\{n_k\},\{m_k\}$,$\{i_k\},\{j_k\}$ such that
$$
\frac{q^{n_1j_1+\cdots +n_kj_k}}{b^{(i_1+\cdots +i_k+m_1j_1+\cdots +m_kj_k)t}}\leq s <\frac{q^{n_1j_1+\cdots +n_k(j_k+1)}}{b^{(i_1+\cdots +i_k+m_1j_1+\cdots +m_k(j_k+1))t}}
$$
for $k\geq 1$; if $t=\frac{n_0\log q}{m_0\log b}$ for some $n_0, m_0\in \Bbb Z^+$, there exist sequences $\{n_k\},\{m_k\}$,$\{i_k\},\{j_k\}$ such that
$$
\frac{(q^{m_0}-1)^{n_1j_1+\cdots +n_kj_k}}{b^{(i_1+\cdots +i_k+m_1j_1+\cdots +m_kj_k)m_0t}}\leq s <\frac{(q^{m_0}-1)^{n_1j_1+\cdots +n_k(j_k+1)}}{b^{(i_1+\cdots +i_k+m_1j_1+\cdots +m_k(j_k+1))m_0t}}
$$
for $k\geq 1$;
\par \noindent (ii)For any $t\in(0, \frac{\log q}{\log b})$, if $\frac{t\log b}{\log q}\in \Bbb Q^c$, there exist sequences $\{n_k\},\{m_k\}$,$\{i_k\},\{j_k\}$ such that
\begin{align*}
&\frac{q^{n_1j_1+\cdots +n_kj_k}}{b^{(i_1+\cdots +i_k+m_1j_1+\cdots +m_kj_k)t}}\leq  \frac{1}{i_1+\cdots +i_k+n_1j_1+\cdots +n_kj_k},\\ \quad
&\frac{1}{i_1+\cdots +i_k+n_1j_1+\cdots +n_k(j_k+1)}<\frac{q^{n_1j_1+\cdots +n_k(j_k+1)}}{b^{(i_1+\cdots +i_k+m_1j_1+\cdots +m_k(j_k+1))t}}
\end{align*}
for $k\geq 1$; if $t=\frac{n_0\log q}{m_0\log b}$ for some $n_0, m_0\in \Bbb Z^+$, there exist sequences $\{n_k\},\{m_k\}$,$\{i_k\},\{j_k\}$ such that
\begin{align*}
&\frac{(q^{m_0}-1)^{n_1j_1+\cdots +n_kj_k}}{b^{(i_1+\cdots +i_k+m_1j_1+\cdots +m_kj_k)m_0t}}\leq  \frac{1}{i_1+\cdots +i_k+n_1j_1+\cdots +n_kj_k},\\ \quad
&\frac{1}{i_1+\cdots +i_k+n_1j_1+\cdots +n_k(j_k+1)}<\frac{(q^{m_0}-1)^{n_1j_1+\cdots +n_k(j_k+1)}}{b^{(i_1+\cdots +i_k+m_1j_1+\cdots +m_k(j_k+1))m_0t}}
\end{align*}
for $k\geq 1$;
\par \noindent (iii)For any $t\in(0, \frac{\log q}{\log b})$, if $\frac{t\log b}{\log q}\in \Bbb Q^c$, there exist sequences $\{n_k\},\{m_k\}$,$\{i_k\},\{j_k\}$ such that
\begin{align*}
&\frac{q^{n_1j_1+\cdots +n_kj_k}}{b^{(i_1+\cdots +i_k+m_1j_1+\cdots +m_kj_k)t}}\leq  i_1+\cdots +i_k+n_1j_1+\cdots +n_kj_k\\
&<i_1+\cdots +i_k+n_1j_1+\cdots +n_k(j_k+1)<\frac{q^{n_1j_1+\cdots +n_k(j_k+1)}}{b^{(i_1+\cdots +i_k+m_1j_1+\cdots +m_k(j_k+1))t}}
\end{align*}
for $k\geq 1$; if $t=\frac{n_0\log q}{m_0\log b}$ for some $n_0, m_0\in \Bbb Z^+$, there exist sequences $\{n_k\},\{m_k\}$,$\{i_k\},\{j_k\}$ such that
\begin{align*}
&\frac{(q^{m_0}-1)^{n_1j_1+\cdots +n_kj_k}}{b^{(i_1+\cdots +i_k+m_1j_1+\cdots +m_kj_k)m_0t}}\leq  i_1+\cdots +i_k+n_1j_1+\cdots +n_kj_k\\
&<i_1+\cdots +i_k+n_1j_1+\cdots +n_k(j_k+1)<\frac{(q^{m_0}-1)^{n_1j_1+\cdots +n_k(j_k+1)}}{b^{(i_1+\cdots +i_k+m_1j_1+\cdots +m_k(j_k+1))m_0t}}
\end{align*}
for $k\geq 1$.
\end{lemma}
\begin{proof}
For $t\in(0, \frac{\log q}{\log b})$, there exists $\alpha\in (0,1)$ such that $t=\alpha \frac{\log q}{\log b}$. If $\alpha\in \Bbb Q^c$, from Dirichlet theorem, it follows that there are two positive integer sequences  $\{n_k\},\{m_k\}$ so that $\Big|\frac{n_k}{m_k}-\alpha\Big|<\frac{1}{m_k^2}$. Moreover, we can take $m_k>n_k, \frac{n_k}{m_k}>\alpha, \lim_{k\to\infty}m_k=+\infty$. If $\alpha=\frac{n_0}{m_0}$, we have $t=\alpha'\frac{\log (q^{m_0}-1)}{\log b^{m_0}}$ with $\alpha'\in (0,1)\cap \Bbb Q^c$. By Dirichlet theorem again, we obtain two positive integer sequences, still denoted by  $\{n_k\},\{m_k\}$, so that $\Big|\frac{n_k}{m_k}-\alpha'\Big|<\frac{1}{m_k^2}$, with the conditions $m_k>n_k, \frac{n_k}{m_k}>\alpha', \lim_{k\to\infty}m_k=+\infty$.   \par
(i) We firstly consider $t=\alpha \frac{\log q}{\log b}$ with  $\alpha\in (0,1)\cap \Bbb Q^c$. It is easy to see that there exists $k_0$ such that for any $i_1\geq k_0, 0=\frac{\log 1}{i_1\log b}<t+\frac{\log s}{i_1\log b}$. Note that the function $f(x)=\frac{n_1x\log q}{(i_1+m_1x)\log b}$ is strictly increasing for $x>0$ and satisfies $\lim\limits_{x\to\infty}f(x)=\frac{n_1\log q}{m_1\log b}>t$. Meanwhile, $\frac{\log s}{(i_1+m_1x)\log b}$ is monotone in $x$. Thus there exists a unique $x_1>0$ such that
$$
\frac{n_1x_1\log q}{(i_1+m_1x_1)\log b}=t+\frac{\log s}{(i_1+m_1x_1)\log b}.
$$
Solving the equality yields  $x_1=\frac{ti_1\log b+\log s}{n_1\log q-m_1t\log b}$, which increases with $i_1$.  Set $j_1=\lfloor x_1\rfloor$. By the definition of floor function, we obtain
$$
\frac{n_1j_1\log q}{(i_1+m_1j_1)\log b}\leq t+\frac{\log s}{(i_1+m_1j_1)\log b}, \, t+\frac{\log s}{(i_1+m_1(j_1+1))\log b}<\frac{n_1(j_1+1)\log q}{(i_1+m_1(j_1+1))\log b}.
$$
It follows that
$$
\frac{q^{n_1j_1}}{b^{(i_1+m_1j_1)t}}\leq s<\frac{q^{n_1(j_1+1)}}{b^{(i_1+m_1(j_1+1))t}}.
$$
We claim that $\frac{n_1j_1\log q}{(i_1+m_1j_1+i_2)\log b}\leq t+\frac{\log s}{(i_1+m_1j_1+i_2)\log b}$ for any $i_2\geq 1$. In fact, if $0<s<1$, then
$$
\frac{n_1j_1\log q}{(i_1+m_1j_1+i_2)\log b}\leq \frac{n_1j_1\log q}{(i_1+m_1j_1)\log b}\leq t+\frac{\log s}{(i_1+m_1j_1)\log b}<t+\frac{\log s}{(i_1+m_1j_1+i_2)\log b}.
$$
If $s\geq 1$, then for $j_1>\frac{\log s}{n_1\log q}$ ( $j_1$ can be large enough ), we have
$$
\Big(\frac{n_1j_1}{i_1+n_1j_1}-\frac{n_1j_1}{i_1+m_1j_1+i_2}\Big)\frac{\log q}{\log b}>\Big(\frac{1}{i_1+m_1j_1}-\frac{1}{i_1+m_1j_1+i_2}\Big)\frac{\log s}{\log b}.
$$
Combining with $\frac{n_1j_1\log q}{(i_1+m_1j_1)\log b}\leq t+\frac{\log s}{(i_1+m_1j_1)\log b}$, we obtain the claim.
By the same monotonicity argument, there exists $x_2>0$ such that
$$
\frac{(n_1j_1+n_2x_2)\log q}{(i_1+m_1j_1+i_2+m_2x_2)\log b}=t+\frac{\log s}{(i_1+m_1j_1+i_2+m_2x_2)\log b}.
$$
The explicit solution  $x_2=\frac{t(i_1+m_1j_1+i_2)\log b+\log s-n_1j_1\log q}{n_2\log q-m_2t\log b}$ is increasing in $i_2$. Let $j_2=\lfloor x_2\rfloor$. Then  \begin{align*}
&\frac{(n_1j_1+n_2j_2)\log q}{(i_1+i_2+m_1j_1+m_2j_2)\log b}\leq t+\frac{\log s}{(i_1+i_2+m_1j_1+m_2j_2)\log b},\\
&t+\frac{\log s}{(i_1+i_2+m_1j_1+m_2(j_2+1))\log b}<\frac{(n_1j_1+n_2(j_2+1))\log q}{(i_1+i_2+m_1j_1+m_2(j_2+1))\log b}.
\end{align*}
which implies
$$
\frac{q^{n_1j_1+n_2j_2}}{b^{(i_1+i_2+m_1j_1+m_2j_2)t}}\leq s<\frac{q^{n_1j_1+n_2(j_2+1)}}{b^{(i_1+i_2+m_1j_1+m_2(j_2+1))t}}.
$$
Iterating the above procedure inductively, for any $k\geq 1$, we can construct integer sequences $\{i_l\}_{l=1}^k, \{j_l\}_{l=1}^k$ satisfying
\begin{align*}
&\frac{(n_1j_1+\cdots +n_kj_k)\log q}{(i_1+\cdots +i_k+m_1j_1+\cdots +m_kj_k)\log b}\leq t+\frac{\log s}{(i_1+\cdots +i_k+m_1j_1+\cdots +m_kj_k)\log b},\\
&t+\frac{\log s}{(i_1+\cdots +i_k+m_1j_1+\cdots +m_k(j_k+1))\log b}<\frac{(n_1j_1+\cdots +n_k(j_k+1))\log q}{(i_1+\cdots +i_k+m_1j_1+\cdots +m_k(j_k+1))\log b},
\end{align*}
This is exactly the required inequality for step $k$. Moreover, the inductive prerequisite is always valid for the next iteration. By taking the limiting construction as $k\to\infty$, we obtain the desired infinite sequences.  For $t=\alpha'\frac{\log (q^{m_0}-1)}{\log b^{m_0}}$ with $\alpha'\in (0,1)\cap \Bbb Q^c$, by a similar argument, one draws the corresponding conclusion. Hence we complete the proof of the statement (i).

\par (ii) At first, we prove it for $t=\alpha \frac{\log q}{\log b}$ with  $\alpha\in (0,1)\cap \Bbb Q^c$. Note that there exists $k_0$ such that for any $i_1\geq k_0, 0=\frac{\log 1}{i_1\log b}<t-\frac{\log i_1}{i_1\log b}$. The function $\frac{n_1x\log q}{(i_1+m_1x)\log b}$ is strictly increasing in $x>1$ with limit $\frac{n_1\log q}{m_1\log b}>t$, while $\frac{\log (i_1+n_1x)}{(i_1+m_1x)\log b}$ is also monotone in $x$. Therefore, there exists $x_1>0$ such that
$$
\frac{n_1x_1\log q}{(i_1+m_1x_1)\log b}=t-\frac{\log (i_1+n_1x_1)}{(i_1+m_1x_1)\log b}.
$$
i.e., $x_1$ satisfies the equality
$$
(n_1\log q-m_1t\log b)x_1=i_1t\log b-\log (i_1+n_1x_1).
$$
It is easy to check that the solution $x_1$ increases with $i_1$ and tends to infinity as $i_1\to\infty$. Let$j_1=\lfloor x_1\rfloor$, then
$$
\frac{n_1j_1\log q}{(i_1+m_1j_1)\log b}\leq t-\frac{\log (i_1+n_1j_1)}{(i_1+m_1j_1)\log b}, \, t-\frac{\log (i_1+n_1(j_1+1))}{(i_1+n_1(j_1+1))\log b}<\frac{(n_1(j_1+1))\log q}{(i_1+n_1(j_1+1))\log b}.
$$
After rearrangement and exponentiation, we obtain
$$
\frac{q^{n_1j_1}}{b^{(i_1+n_1j_1)t}}\leq  \frac{1}{i_1+n_1j_1}, \frac{1}{i_1+n_1(j_1+1)}<\frac{q^{n_1(j_1+1)}}{b^{(i_1+m_1(j_1+1))t}}.
$$
For any $i_2\geq 1$, we have the monotonicity estimate
$$
\frac{n_1j_1\log q}{(i_1+m_1j_1+i_2)\log b}<\frac{n_1j_1\log q}{(i_1+m_1j_1)\log b}\leq t-\frac{\log (i_1+n_1j_1)}{(i_1+m_1j_1)\log b}<t-\frac{\log (i_1+n_1j_1+i_2)}{(i_1+m_1j_1+i_2)\log b},
$$
which guarantees the feasibility of the next-step construction. We then find $x_2>0$ satisfying
$$
\frac{(n_1j_1+n_2x_2)\log q}{(i_1+m_1j_1+i_2+m_2x_2)\log b}=t-\frac{\log (i_1+n_1j_1+i_2+n_2x_2)}{(i_1+m_1j_1+i_2+m_2x_2)\log b}.
$$
We can also observe that $x_2$ increases with $i_2$ and tends to $+\infty$. Taking $j_2=\lfloor x_2\rfloor$, we derive the two-step estimate
$$
\frac{q^{n_1j_1+n_2j_2}}{b^{(i_1+i_2+m_1j_1+m_2j_2)t}}\leq \frac{1}{i_1+i_2+n_1j_1+n_2j_2}, \frac{1}{i_1+i_2+n_1j_1+n_2(j_2+1)}<\frac{q^{n_1j_1+n_2(j_2+1)}}{b^{(i_1+i_2+m_1j_1+m_2(j_2+1))t}}.
$$
By induction, for all $k\geq 1$, we can construct finite sequences $\{i_l\}_{l=1}^k,\{j_l\}_{l=1}^k$ such that
\begin{align*}
&\frac{(n_1j_1+\cdots +n_kj_k)\log q}{(i_1+\cdots +i_k+m_1j_1+\cdots +m_kj_k)\log b}\leq t-\frac{\log (i_1+\cdots +i_k+n_1j_1+\cdots +n_kj_k)}{(i_1+\cdots +i_k+m_1j_1+\cdots +m_kj_k)\log b},\\
&t-\frac{\log (i_1+\cdots +i_k+n_1j_1+\cdots +n_k(j_k+1))}{(i_1+\cdots +i_k+m_1j_1+\cdots +m_k(j_k+1))\log b}<\frac{(n_1j_1+\cdots +n_k(j_k+1))\log q}{(i_1+\cdots +i_k+m_1j_1+\cdots +m_k(j_k+1))\log b}.
\end{align*}
This yields exactly the required double inequality for each $k$. The inductive condition persists for all subsequent iterations. Passing to the infinite limit, we obtain the desired infinite sequences. For $t=\alpha'\frac{\log (q^{m_0}-1)}{\log b^{m_0}}$ with $\alpha'\in (0,1)\cap \Bbb Q^c$, through a similar argument, we arrive at the corresponding conclusion, finishing the proof of (ii).
\par (iii) If $t=\alpha \frac{\log q}{\log b}$ with  $\alpha\in (0,1)\cap \Bbb Q^c$, then there exists $k_0$ such that for any $i_1\geq k_0, 0=\frac{\log 1}{i_1\log b}<t+\frac{\log i_1}{i_1\log b}<\frac{\log q}{\log b}$. Again, $\frac{n_1x\log q}{(i_1+m_1x)\log b}$ is strictly increasing in $x>1$ with limit greater than $t$, and $\frac{\log (i_1+n_1x)}{(i_1+m_1x)\log b}$ is monotone in $x$. Thus there exists $x_1>0$ such that
$$
\frac{n_1x_1\log q}{(i_1+n_1x_1)\log b}=t+\frac{\log (i_1+n_1x_1)}{(i_1+m_1x_1)\log b},
$$
or equivalently, $(n_1\log q-m_1t\log b)x_1=i_1t\log b+\log (i_1+m_1x_1)$. The solution $x_1$ increases with $i_1$ and tends to infinity. Set $j_1=\lfloor x_1\rfloor$, then
$$
\frac{n_1j_1\log q}{(i_1+m_1j_1)\log b}\leq t+\frac{\log (i_1+n_1j_1)}{(i_1+m_1j_1)\log b}, \quad t+\frac{\log (i_1+n_1(j_1+1))}{(i_1+m_1j_1+1)\log b}<\frac{n_1(j_1+1)\log q}{(i_1+m_1j_1+1)\log b}.
$$
Rearranging gives the first-step chain inequality
$$
\frac{q^{n_1j_1}}{b^{(i_1+m_1j_1)t}}\leq  i_1+n_1j_1<i_1+n_1(j_1+1)<\frac{q^{n_1(j_1+1)}}{b^{(i_1+m_1(j_1+1))t}}.
$$
We verify the inductive condition: for any $i_2\geq 1$,
$$
\frac{n_1j_1\log q}{(i_1+m_1j_1+i_2)\log b}\leq t+\frac{\log (i_1+n_1j_1+i_2)}{(i_1+m_1j_1+i_2)\log b}.
$$
This follows from the gap estimate
\begin{align*}
\Big(\frac{n_1j_1}{i_1+m_1j_1}-\frac{n_1j_1}{i_1+m_1j_1+i_2}\Big)\frac{\log q}{\log b}&>\frac{\log (i_1+n_1j_1)}{(i_1+m_1j_1)\log b}-\frac{\log (i_1+n_1j_1)}{(i_1+m_1j_1+i_2)\log b}\\
&>\frac{\log (i_1+n_1j_1)}{(i_1+m_1j_1)\log b}-\frac{\log (i_1+n_1j_1+i_2)}{(i_1+m_1j_1+i_2)\log b},
\end{align*}
combined with the former inequality for $i_1, j_1$. Repeating the monotonicity argument, we find $x_2>0$ and set $j_2=\lfloor x_2\rfloor$, which yields the two-step chain inequality. By induction on $k$, for every positive integer $k$, we can construct sequences $\{i_l\}_{l=1}^k,\{j_l\}_{l=1}^k$ such that
\begin{align*}
&\frac{(n_1j_1+\cdots +n_kj_k)\log q}{(i_1+\cdots +i_k+m_1j_1+\cdots +m_kj_k)\log b}\leq t+\frac{\log (i_1+\cdots +i_k+n_1j_1+\cdots +n_kj_k)}{(i_1+\cdots +i_k+m_1j_1+\cdots +m_kj_k)\log b},\\
&t+\frac{\log (i_1+\cdots +i_k+n_1j_1+\cdots +n_k(j_k+1))}{(i_1+\cdots +i_k+m_1j_1+\cdots +m_k(j_k+1))\log b}<\frac{(n_1j_1+\cdots +n_k(j_k+1))\log q}{(i_1+\cdots +i_k+m_1j_1+\cdots +m_k(j_k+1))\log b}.
\end{align*}
The inductive construction can be extended to all $k\in\mathbb{N}$, and the limiting infinite sequences satisfy all required properties. If $t=\alpha'\frac{\log (q^{m_0}-1)}{\log b^{m_0}}$ with $\alpha'\in (0,1)\cap \Bbb Q^c$, we can obtain the associated conclusion by a similar discussion. This completes the proof of (iii) and the whole lemma.
\end{proof}

\begin{remark}\label{rem2.7}
From the proof of Lemma \ref{lem2.6}, it is easy to see that for any $t\in(0, \frac{\log q}{\log b})$ and $s\in [0,\infty]$ and the associated sequences $\{n_k\},\{m_k\},\{i_k\},\{j_k\}$, we have
\begin{align}\label{2.1}
&\lim_{k\to \infty}\frac{(n_1j_1+\cdots +n_kj_k)\log q}{(i_1+\cdots +i_{k}+m_1j_1+\cdots +m_kj_k)\log b}=t,
\quad \lim_{k\to \infty}\frac{q^{n_1j_1+\cdots +n_kj_k}}{b^{(i_1+\cdots +i_{k}+m_1j_1+\cdots +m_kj_k)t}}=s,\nonumber \\
&\lim_{k\to \infty}\frac{(n_1j_1+\cdots +n_kj_k)\log (q^{m_0}-1)}{(i_1+\cdots +i_{k}+m_1j_1+\cdots +m_kj_k)\log b^{m_0}}=t,
\quad \lim_{k\to \infty}\frac{(q^{m_0}-1)^{n_1j_1+\cdots +n_kj_k}}{b^{(i_1+\cdots +i_{k}+m_1j_1+\cdots +m_kj_k)m_0t}}=s,
\end{align}
Moreover, we can take $i_k$ arbitrarily large and $\lim_{k\to \infty}i_k=\lim_{k\to \infty}j_k=+\infty$.
\end{remark}

For convenience, we denote $\mathcal{I}_k=\sum_{l=1}^ki_l,\mathcal{J}_k=\sum_{l=1}^kn_lj_l, \mathcal{J}'_k=\sum_{l=1}^km_lj_l$, then \eqref{2.1} becomes
\begin{align}\label{2.2}
&\lim_{k\to \infty}\frac{\mathcal{J}_k\log q}{(\mathcal{I}_k+\mathcal{J}'_k)\log b}=t, \quad \lim_{k\to \infty}\frac{q^{\mathcal{J}_k}}{b^{(\mathcal{I}_k+\mathcal{J}'_k)t}}=s,\nonumber \\
&\lim_{k\to \infty}\frac{\mathcal{J}_k\log (q^{m_0}-1)}{(\mathcal{I}_k+\mathcal{J}'_k)m_0\log b}=t, \quad \lim_{k\to \infty}\frac{(q^{m_0}-1)^{\mathcal{J}_k}}{b^{(\mathcal{I}_k+\mathcal{J}'_k)m_0t}}=s.
\end{align}

\section{\bf Proofs of Theorems 1.3 and 1.4 \label{sect.3}}

In this section, we will prove Theorem \ref{th1.3} and Theorem \ref{th1.4}. We firstly construct the spectrum of $\mu$. Lemma \ref{lem2.2} guarantees the existence of $\epsilon_1\in (0,\frac12)$ and an integer $N_{1}\ge 1$ such that for any $\xi\in [0,1]$, there exist integers $\ell_1^{(\xi)},\ell_2^{(\xi)}\in \{0,1,\dots,b^{N_1}-1\}$ with $\ell_1^{(\xi)}<\ell_2^{(\xi)}$
satisfying
\[
\bigl|\widehat{\mu}\bigl(\xi+\ell_i^{(\xi)}\bigr)\bigr|\ge 2\epsilon_1.
\]
By the uniform continuity of $\widehat{\mu}(\xi)$, there exists $\delta_1\in (0,1)$ such that for any $|x-\xi|<\delta_1$,
\[
\bigl|\widehat{\mu}\bigl(x+\ell_i^{(\xi)}\bigr)-\widehat{\mu}\bigl(\xi+\ell_i^{(\xi)}\bigr)\bigr|<\epsilon_1,
\]
which implies that
\begin{align}\label{3.1}
\bigl|\widehat{\mu}\bigl(x+\ell_i^{(\xi)}\bigr)\bigr
|
\ge
\bigl|\widehat{\mu}\bigl(\xi+\ell_i^{(\xi)}\bigr)\bigr
|
-
\bigl|\widehat{\mu}\bigl(x+\ell_i^{(\xi)}\bigr)-\widehat{\mu}\bigl(\xi+\ell_i^{(\xi)}\bigr)\bigr
|
\ge \epsilon_1.
\end{align}
We denote a particular non-zero
$\ell_j^{(0)}$ as $\tilde{l}$, satisfying
\[
\bigl|\widehat{\mu}\bigl(\xi+ \tilde{l}\bigr)\bigr|\ge \epsilon_1\quad \text{for any } |\xi|<\delta_1.
\]

\par Let $\Xi=\{0,1,\cdots, q-1\}, L=\{0=l_0,l_1,\cdots, l_{q-1}\}, \Theta=\{0\}$ and denote
$$
\Xi^k:=\underbrace{\Xi\times \cdots \times \Xi}_{k}, \;\Theta^n \Xi^k=\underbrace{\Theta \times \cdots \times \Theta}_{n}\times \underbrace{\Xi\times \cdots \times \Xi}_{k}\subset \Xi^{n+k}.
$$
Define the map $T: \Xi \to L$ as $T(i)=l_i$. For $n \ge 1$, denote
\begin{align*}
\mathcal{L}_n := L + bL + \dots + b^{n-1}L,
\end{align*}
with the natural mapping $T_n: \Xi^n \to \mathcal{L}_n$ given by
\begin{align*}
T_n(i_1i_2\cdots i_n) := \sum_{j=1}^{n} T(i_j) b^{j-1}.
\end{align*}
For any  $t=\alpha \frac{\log q}{\log b}$ with  $\alpha\in (0,1)\cap \Bbb Q^c$ and $s\in [0,\infty]$, let $\{n_k\},\{m_k\},\{i_k\},\{j_k\}$ be the corresponding sequences given by Lemma \ref{lem2.6}, define $\Omega_k=\Theta^{i_k}\Xi^{(m_k-n_k)j_k}\Theta^{n_kj_k}$ and
\[
\phi_{k}(l,j):=
\begin{cases}
0 & l\in T_{i_k+m_kj_k}(\Omega_k); \\
\ell_j^{(b^{-(i_k+m_kj_k)}l)}& l\in\mathcal{L}_{i_k+m_kj_k}\setminus T_{i_k+m_kj_k}(\Omega_k),
\end{cases} \quad j=1,2
\]
and
\begin{equation}\label{3.2}
\psi_{k}(l,j):=l+ b^{i_k+m_kj_k}\phi_{k}(l,j).
\end{equation}
It is easy to see that $\ell_j^{(b^{-(i_k+m_kj_k)}l)}<b^{N_1}$. Note $\hat{\mu}(0)=1$. Then there exist $\epsilon_2, \delta_2$ such that $|\hat{\mu}(y)|\ge \epsilon_2$ for $|y|\le \delta_2$.  Since $\lim_{k\to\infty}j_k=+\infty$, then for $k$ large enough and $l\in T_{i_k+m_kj_k}(\Omega_k)$, we have
$$
b^{-(i_k+m_kj_k)}l\leq b^{-(i_k+m_kj_k)}(1+b+\cdots +b^{i_k+(m_k-n_k)j_k-1})\max_{1\leq j\leq q-1}l_j \leq b^{-n_kj_k}<\frac{\delta_2}{2}.
$$
This combine with \eqref{3.1}-\eqref{3.2} imply
\begin{equation}\label{3.3}
\bigl|\widehat{\mu}\bigl(y + b^{-(i_k+m_kj_k)}\psi_{k}(l,j)\bigr)\bigr|
\ge \epsilon:=\min\{\epsilon_2,\epsilon_2\}
\quad\text{when } |y|<\delta=\min\Big\{\frac{\delta_1}{2},\frac{\delta_2}{2}\Big\}.
\end{equation}

We now construct the spectrum $\Lambda^{(t,s)}(I)$ inductively for each $\omega= \omega_1\omega_2\cdots \in \{1,2\}^\infty$. We begin with the initial sets at the first iterative step. Define
\[
\widetilde{\Lambda}^{(t,s)}_{1}(\omega):= \bigl\{\psi_{1}(l,\omega_1) : l \in T_{i_1+m_1j_1}(\Omega_1)\bigr\}
\]
and
\[
\Gamma_{1}(I) := \bigl\{\psi_{1}(l,\omega_1) : l \in \mathcal{L}_{i_1+m_1j_1}\setminus T_{i_1+m_1j_1}(\Omega_1)\bigr\}
= \bigl\{\lambda_1^{(1)},\lambda_2^{(1)},\dots,\lambda_{\sigma_1}^{(1)}\bigr\},\quad
\]
where the mapping $\psi_{1}(l,i_1)$ is defined in \eqref{3.2}. Let $N\geq N_1$ be a positive integer sufficiently large such that $b^{-2N}<\delta'$. We further define the augmented component of the spectrum at the first level
\[
\overline{\Lambda}^{(t,s)}_{1}(\omega) :=
\Bigl\{\lambda_j^{(1)} + b^{\mathcal{I}_1+\mathcal{J}'_1+(N+2)(j+1)}\tilde{l} : 1 \le j \le \sigma_1\Bigr\},
\]
Next, since $i_2$ can be taken arbitrarily large, we choose $i_2+m_2j_2$ sufficiently large to satisfy the separation condition
\[
b^{-(i_2+m_2j_2)+(N+2)(\sigma_1+2)} < \delta.
\]
We then define the second-stage iterative sets. Let
\[
\widetilde{\Lambda}^{(t,s)}_{2}(\omega):=
\widetilde{\Lambda}^{(t,s)}_{1}(\omega)
+ b^{\mathcal{I}_1+\mathcal{J}'_1}\bigl\{\psi_{2}(l,\omega_2): l\in T_{i_2+m_2j_2}(\Omega_2)\bigr\}
\]
and decompose the complementary set as
\[
\Gamma_{2}(\omega):=
\Gamma^{(1)}_{2}(\omega)\cup \Gamma^{(2)}_{2}(\omega)
=\bigl\{\lambda^{(2)}_1,\dots,\lambda^{(2)}_{\sigma_2}\bigr\},
\]
where
\[
\Gamma^{(1)}_{2}(\omega)
=\widetilde{\Lambda}^{(t,s)}_{1}(\omega)
+ b^{\mathcal{I}_1+\mathcal{J}'_1}\bigl\{\psi_{2}(l,\omega_2): l\in  \mathcal{L}_{i_2+m_2j_2}\setminus T_{i_2+m_2j_2}(\Omega_{2})\bigr\}
\]
and
\[
\Gamma^{(2)}_{2}(\omega)
=\overline{\Lambda}^{(t,s)}_{1}(\omega)
+ b^{\mathcal{I}_1+\mathcal{J}'_1}\bigl\{\psi_{2}(l,\omega_2): l\in \mathcal{L}_{i_2+m_2j_2}\setminus\{0\}\bigr\}.
\]
Accordingly, the updated augmented spectrum at the second level is given by
\[
\overline{\Lambda}^{(t,s)}_{2}(I):=
\overline{\Lambda}^{(t,s)}_{1}(I)
\cup \Bigl\{\lambda^{(2)}_j + b^{\mathcal{I}_2+\mathcal{J}'_2+(2+N)(j+1)}\tilde{l}
: 1\le j\le \sigma_2\Bigr\}.
\]
We now proceed to the general inductive step. For each $k\geq 3$, we choose $i_k+m_2j_k$ sufficiently large such that
\begin{equation}\label{3.4}
b^{-(i_k+m_kj_k)+(N+2)(\sigma_{k-1}+2)} <\delta.
\end{equation}
We recursively define the core iterative set at stage $k$ by
\begin{equation}\label{eq3.5}
\widetilde{\Lambda}^{(t,s)}_{k}(\omega):=
\widetilde{\Lambda}^{(t,s)}_{k-1}(\omega)
+ b^{\mathcal{I}_{k-1}+\mathcal{J}'_{k-1}}\bigl\{\psi_{k}(l,\omega_k): l\in T_{i_k+m_kj_k}(\Omega_{k})\bigr\}.
\end{equation}
Similarly, we decompose the complementary set
\[
\Gamma_{k}(\omega):=
\Gamma^{(1)}_{k}(\omega)\cup \Gamma^{(2)}_{k}(\omega)
=\bigl\{\lambda^{(k)}_1,\dots,\lambda^{(k)}_{\sigma_k}\bigr\},
\]
with
\[
\Gamma^{(1)}_{k}(\omega)
=\widetilde{\Lambda}^{(t,s)}_{k-1}(\omega)
+ b^{\mathcal{I}_{k-1}+\mathcal{J}'_{k-1}}\bigl\{\psi_{k}(l,\omega_k): l\in \mathcal{L}_{i_k+m_kj_k}\setminus T_{i_k+m_kj_k}(\Omega_{k})\bigr\}
\]
and
\[
\Gamma^{(2)}_{k}(\omega)
=\overline{\Lambda}^{(t,s)}_{k-1}(\omega)
+ b^{\mathcal{I}_{k-1}+\mathcal{J}'_{k-1}}\bigl\{\psi_{k}(l,\omega_k): l\in \mathcal{L}_{i_k+m_kj_k}\setminus\{0\}\bigr\}.
\]
The augmented spectrum at level $k$ is then updated as
\begin{equation}\label{3.6}
\overline{\Lambda}^{(t,s)}_{k}(\omega):=
\overline{\Lambda}^{(t,s)}_{k-1}(\omega)
\cup \Bigl\{\lambda^{(k)}_j + b^{\mathcal{I}_{k}+\mathcal{J}'_{k}+(j+1)(2+N)}\tilde{l}
: 1\le j\le \sigma_k\Bigr\}.
\end{equation}
In this construction, $\widetilde{\Lambda}^{(t,s)}_{k}(\omega)$ is designed to precisely realize the prescribed Beurling dimension and Beurling density, while the auxiliary set $\overline{\Lambda}^{(t,s)}_{k}(\omega)$ compensates the core set to form a complete valid spectrum. By construction, the sequences of sets are strictly increasing
\begin{equation}\label{3.7}
\overline{\Lambda}^{(t,s)}_{k-1}(\omega) \subset \overline{\Lambda}^{(t,s)}_{k}(\omega),\quad
\widetilde{\Lambda}^{(t,s)}_{k-1}(\omega) \subset \widetilde{\Lambda}^{(t,s)}_{k}(\omega),
\quad \forall\ k\ge 2.
\end{equation}
Finally, we take the increasing limits of these iterative families to obtain the final spectral components and the full spectrum
\begin{equation}\label{3.8}
\widetilde{\Lambda}^{(t,s)}(\omega):=\bigcup_{k=1}^{\infty}\widetilde{\Lambda}^{(t,s)}_{k}(\omega),\quad
\overline{\Lambda}^{(t,s)}(\omega):=\bigcup_{k=1}^{\infty}\overline{\Lambda}^{(t,s)}_{k}(\omega),\quad
\Lambda^{(t,s)}(\omega):=\widetilde{\Lambda}^{(t,s)}(\omega)\cup\overline{\Lambda}^{(t,s)}(\omega).
\end{equation}

The above iterative construction yields crucial quantitative estimates for the constructed spectral sets, which we summarize in the following lemma.

\begin{lemma}\label{lem3.1}
As the above notations, we have
\[
0\le \lambda < 2b^{\mathcal{I}_{k}+\mathcal{J}'_{k}+N},\quad \forall\ \lambda\in \Gamma_{k}(\omega)\cup \widetilde{\Lambda}^{(t,s)}_{k}(\omega),
\]
and
\[
b^{\mathcal{I}_{k}+\mathcal{J}'_{k}+2(N+2)}\le \lambda < b^{\mathcal{I}_{k}+\mathcal{J}'_{k}+(N+2)(\sigma_k+2)},\quad
\forall\ \lambda\in \overline{\Lambda}^{(t,s)}_{k}(\omega)\setminus \overline{\Lambda}^{(t,s)}_{k-1}(\omega).
\]
Moreover, the set $\overline{\Lambda}^{(t,s)}(\omega)$ is $b$-lacunary. We can further see that for $k$ large enough,
\[
0\le \lambda < b^{\mathcal{I}_{k}+\mathcal{J}'_{k}},\quad \forall\ \lambda\in \widetilde{\Lambda}^{(t,s)}_{k}(\omega).
\]
\end{lemma}
\begin{proof}
We proceed with mathematical induction to verify the above estimates. For the initial case $k=1$, recall that $\phi_{1}(l,j)< b^{N_1}\le b^N$ by definition. A direct verification implies
\[
0\le \lambda < b^{i_1+m_1j_1}+ b^{i_1+m_1j_1}\cdot b^N< 2b^{\mathcal{I}_1+\mathcal{J}'_1+N},\quad \forall\ \lambda\in \Gamma_{1}(\omega)\cup \widetilde{\Lambda}^{(t,s)}_{1}(\omega).
\]
For elements $\lambda\in \overline{\Lambda}^{(t,s)}_{1}(\omega)$, since $1\le \tilde{l}<b^N$, we derive the upper and lower bounds
\[
b^{\mathcal{I}_1+\mathcal{J}'_1+2(N+2)}\le \lambda < 2b^{\mathcal{I}_1+\mathcal{J}'_1+N}+b^{\mathcal{I}_1+\mathcal{J}'_1+(2+N)(\sigma_1+1)+N}
< b^{\mathcal{I}_1+\mathcal{J}'_1+(N+2)(\sigma_1+2)}.
\]
Assume that all conclusions of the lemma hold for all integers up to $k-1$. For the inductive step, consider arbitrary $\lambda\in \Gamma_{k}(\omega)\cup \widetilde{\Lambda}^{(t,s)}_{k}(\omega)$. By the iterative construction, such $\lambda$ admits the decomposition
\[
\lambda=\lambda'+b^{\mathcal{I}_{k-1}+\mathcal{J}'_{k-1}}\psi_{k}(l,\omega_k),
\]
where $\lambda'\in \widetilde{\Lambda}^{(t,s)}_{k-1}(\omega)\cup \overline{\Lambda}^{(t,s)}_{k-1}(\omega)$ and $l\in \mathcal{L}_{i_k+m_kj_k}$. Combining the induction hypothesis and the uniform bound $0\le \psi_{k}(l,\omega_k)< b^{i_k+m_kj_k}+b^{i_k+m_kj_k+N}$, we obtain
\[
0\le \lambda < 2b^{\mathcal{I}_{k-1}+\mathcal{J}'_{k-1}+N}+b^{\mathcal{I}_{k-1}+\mathcal{J}'_{k-1}}(b^{i_k+m_kj_k}+b^{i_k+m_kj_k+N})<2b^{\mathcal{I}_{k}+\mathcal{J}'_{k}+N}.
\]
Further, for $k$ large enough and any $\lambda\in \widetilde{\Lambda}^{(t,s)}_{k}(\omega)$, we have
$$
0\le \lambda < 2b^{\mathcal{I}_{k-1}+\mathcal{J}'_{k-1}+N}+b^{\mathcal{I}_{k-1}+\mathcal{J}'_{k-1}}\cdot b^{i_k+(m_k-n_k)j_k}<b^{\mathcal{I}_{k}+\mathcal{J}'_{k}},
$$
as $\lim_{k\to\infty}j_k=\lim_{k\to\infty}m_k=+\infty$.  For the newly added elements at the $k$-th stage, i.e., $\lambda\in \overline{\Lambda}^{(t,s)}_{k}(\omega)\setminus \overline{\Lambda}^{(t,s)}_{k-1}(\omega)$, we similarly establish
\[
b^{\mathcal{I}_{k}+\mathcal{J}'_{k}+2(N+2)}\le \lambda < 2b^{\mathcal{I}_{k}+\mathcal{J}'_{k}+N}+b^{\mathcal{I}_{k}+\mathcal{J}'_{k}+(N+2)(\sigma_k+1)+N}
< b^{\mathcal{I}_{k}+\mathcal{J}'_{k}+(N+2)(\sigma_k+2)},
\]
which completes the induction for the estimation.

Next, we prove the $b$-lacunary property of $\overline{\Lambda}^{(t,s)}(\omega)$. Take any two elements $\lambda_1,\lambda_2\in \overline{\Lambda}^{(t,s)}(\omega)$ satisfying $\lambda_1>\lambda_2$. We discuss two exhaustive cases.

\textit{Case 1:}  Suppose  $\lambda_1\in \overline{\Lambda}^{(t,s)}_{k}(\omega)\setminus \overline{\Lambda}^{(t,s)}_{k-1}(\omega)$ and $\lambda_2\in \overline{\Lambda}^{(t,s)}_{k'}(\omega)$ for some $k'<k$, using the conclusion from the first part of this lemma, then
\[
\lambda_1\ge b^{\mathcal{I}_{k}+\mathcal{J}'_{k}+2(N+2)}>b^{\mathcal{I}_{k'+1}+\mathcal{J}'_{k'+1}+1}> b^{\mathcal{I}_{k'}+\mathcal{J}'_{k'}+(N+2)(\sigma_k+2)+1}\ge b\lambda_2.
\]

\textit{Case 2:} Suppose both $\lambda_1,\lambda_2\in \overline{\Lambda}^{(t,s)}_{k}(\omega)\setminus \overline{\Lambda}^{(t,s)}_{k-1}(\omega)$. By construction, there exist indices $r_1,r_2$ and corresponding elements $\lambda^{(k)}_{r_1},\lambda^{(k)}_{r_2}\in \Gamma_{k}(\omega)$ such that
\[
\lambda_i=\lambda^{(k)}_{r_i}+b^{\mathcal{I}_{k}+\mathcal{J}'_{k}+(2+N)(r_i+1)}\tilde{l}.
\]
The assumption $\lambda_1>\lambda_2$ directly implies $r_1>r_2$. We then compute the difference $\lambda_1-b\lambda_2$ as follows
\[
\begin{aligned}
\lambda_1-b\lambda_2
&\ge \bigl(b^{\mathcal{I}_{k}+\mathcal{J}'_{k}+(2+N)(r_1+1)}\tilde{l}-b^{\mathcal{I}_{k}+\mathcal{J}'_{k}+(2+N)(r_2+1)+1}\tilde{l}\bigr)-\bigl|\lambda^{(k)}_{r_1}-b\lambda^{(k)}_{r_2}\bigr| \\
&\ge \bigl(b^{(2+N)(r_1-r_2)}-b\bigr)b^{\mathcal{I}_{k}+\mathcal{J}'_{k}+(2+N)(r_2+1)}\tilde{l}-2(b+1)b^{\mathcal{I}_{k}+\mathcal{J}'_{k}+N} \\
&>0.
\end{aligned}
\]
Combining the above two cases, we conclude that $\lambda_1 \ge b\lambda_2$ holds for all pairs $\lambda_1>\lambda_2$ in $\overline{\Lambda}^{(t,s)}(\omega)$. Therefore, the set $\overline{\Lambda}^{(t,s)}(\omega)$ is $b$-lacunary. This finishes the proof of the lemma.
\end{proof}

\begin{lemma}\label{lem3.2}
Let $\Lambda^{(t,s)}(\omega)$ be defined as in \eqref{3.8}. Then $\Lambda^{(t,s)}(\omega)$ is a spectrum of $\mu$.
\end{lemma}
\begin{proof}
We first verify that the truncated iterative set forms a valid spectrum for the finite convolution approximation of $\mu$. For each integer $k\ge 1$, define the $k$-level truncated spectral set
$\Lambda^{(t,s)}_{k}(\omega):=\widetilde{\Lambda}^{(t,s)}_{k}(\omega)\cup \overline{\Lambda}^{(t,s)}_{k}(\omega)$
and the finite convolution measure
\[
\mu_{\mathcal{I}_{k}+\mathcal{J}'_{k}}:=\delta_{b^{-1}D}*\delta_{b^{-2}D}*\cdots*\delta_{b^{-(\mathcal{I}_{k}+\mathcal{J}'_{k})}D}.
\]
Take any two distinct elements $\lambda_1,\lambda_2\in \Lambda^{(t,s)}_{k}(\omega)$. By the iterative construction of spectral sets, they admit the unified decomposition
\[
\lambda_i=\sum_{r=1}^{k} b^{\mathcal{I}_{r-1}+\mathcal{J}'_{r-1}}\psi_{r}(\ell^{(i)}_r,\omega_r)
+\sum_{r=c_i}^{k} b^{\mathcal{I}_{r}+\mathcal{J}'_{r}+(N+2)(j^{(i)}_r+1)}\tilde{l},
\]
where $1\le j^{(i)}_r\le \sigma_r$, $1\le c_i\le k+1$ ( the second summation vanishes identically when $c_i=k+1$ ) and the sequence $\{\ell^{(i)}_r\}$ satisfies
\[
\ell^{(i)}_r\in
\begin{cases}
T_{i_r+m_rj_r}\big(\Omega_{r}\big), & 1\le r\le c_i-1;\\
\mathcal{L}_{i_r+m_rj_r}, & c_i\le r\le k.
\end{cases}
\]
Let $\tau=\min\big\{1\le r\le k: \ell^{(1)}_r\neq \ell^{(2)}_r\big\}\le k$, which yields that $j^{(1)}_r=j^{(2)}_r$ for $1\le r\le \tau-1$. Direct computation of the difference yields
\[
\lambda_1-\lambda_2\in b^{\mathcal{I}_{\tau-1}+\mathcal{J}'_{\tau-1}}\big(\ell^{(1)}_{\tau}-\ell^{(2)}_{\tau}+b^{i_{\tau}+m_{\tau}j_{\tau}}\mathbb{Z}\big)
\subset b^{\mathcal{I}_{\tau-1}+\mathcal{J}'_{\tau-1}}\big((\mathcal{L}_{i_{\tau}+m_{\tau}j_{\tau}}-\mathcal{L}_{i_{\tau}+m_{\tau}j_{\tau}})\setminus\{0\}+b^{i_{\tau}+m_{\tau}j_{\tau}}\mathbb{Z}\big).
\]
This inclusion implies
$
\lambda_1-\lambda_2\in  \mathcal{Z}\bigl(\widehat{\mu_{\mathcal{I}_{k}+\mathcal{J}'_{k}}}\bigr).
$
Therefore, $\Lambda^{(t,s)}_{k}(\omega)$ is an orthogonal set of $\mu_{\mathcal{I}_{k}+\mathcal{J}'_{k}}$. The cardinality condition $\#\Lambda^{(t,s)}_{k}(\omega)=q^{\mathcal{I}_{k}+\mathcal{J}'_{k}}=\#\operatorname{supp}(\mu_{\mathcal{I}_{k}+\mathcal{J}'_{k}})$ implies that $\Lambda^{(t,s)}_{k}(\omega)$ is indeed a spectrum, where $\operatorname{supp}(\mu_{\mathcal{I}_{k}+\mathcal{J}'_{k}})$ denotes the support of $\mu_{\mathcal{I}_{k}+\mathcal{J}'_{k}}$.

We next establish the completeness of the limiting spectrum $\Lambda^{(t,s)}(\omega)$. For any $\lambda\in \Lambda^{(t,s)}_{k}(\omega)$, the structural construction of the spectral set guarantees the decomposition
\[
\lambda=\lambda'+b^{\mathcal{I}_{k-1}+\mathcal{J}'_{k-1}}\psi_{k}(\ell,\omega_k)+cb^{\mathcal{I}_{k}+\mathcal{J}'_{k}+(N+2)(v_\ell+1)}\tilde{l},
\]
where $\lambda'\in \Lambda^{(t,s)}_{k-1}(\omega)$, $\ell\in \mathcal{L}_{i_k+m_kj_k}$, $1\le v_\ell\le \sigma_{k}$ and  and the binary indicator $c\in\{0,1\}$ distinguishes elements from the core set $\widetilde{\Lambda}^{(t,s)}_{k}(\omega)$ ($c=0$) and the augmented set $\overline{\Lambda}^{(t,s)}_{k}(\omega)$ ($c=1$). It follows from Lemma \ref{lem3.1} and \eqref{3.4} that
\[
\big|b^{-(\mathcal{I}_{k}+\mathcal{J}'_{k})}\lambda'\big|
< \big|b^{-(\mathcal{I}_{k}+\mathcal{J}'_{k})}\cdot b^{\mathcal{I}_{k-1}+\mathcal{J}'_{k-1}+(N+2)(\sigma_{k-1}+2)}\big|
< \delta.
\]
Hence, by \eqref{3.1}-\eqref{3.3},
\[
\begin{aligned}
\big|\widehat{\mu}(b^{-(\mathcal{I}_{k}+\mathcal{J}'_{k})}\lambda)\big|&=\Big|\widehat{\mu}\Big(b^{-(\mathcal{I}_{k}+\mathcal{J}'_{k})}\big(\lambda'
+ b^{(\mathcal{I}_{k-1}+\mathcal{J}'_{k-1})}\psi_{k}(\ell,\omega_{k})
+ cb^{\mathcal{I}_{k}+\mathcal{J}'_{k}+(N+2)(v_\ell+1)}\tilde{l}\big)\Big)\Big| \\
&= \Big|\widehat{\mu_{(N+2)(v_\ell+1)}}\big(b^{-(\mathcal{I}_{k}+\mathcal{J}'_{k})}\lambda'
+ b^{-(i_k+m_kj_k)}\psi_{k}(\ell,\omega_{k})\big)\Big|\\
&\quad \times\big|\widehat{\mu}(b^{-(\mathcal{I}_{k}+\mathcal{J}'_{k}+(N+2)(v_\ell+1))}\lambda' + b^{-(i_k+m_kj_k+(N+2)(v_\ell+1))}\psi_{k}(\ell,\omega_{k})+ c\tilde{l})\big| \\
&\geq \Big|\hat{\mu}\big(b^{-(\mathcal{I}_{k}+\mathcal{J}'_{k})}\lambda'
+ b^{-(i_k+m_kj_k)}\psi_{k}(\ell,\omega_{k})\big)\Big|\\
&\quad \times\big|\widehat{\mu}(b^{-(\mathcal{I}_{k}+\mathcal{J}'_{k}+(N+2)(v_\ell+1))}\lambda' + b^{-(i_k+m_kj_k+(N+2)(v_\ell+1))}\psi_{k}(\ell,\omega_{k})+ c\tilde{l})\big| \\
&\ge \epsilon^2.
\end{aligned}
\]
where
$$
0\leq b^{-(\mathcal{I}_{k}+\mathcal{J}'_{k}+(N+2)(v_\ell+1))}\lambda' + b^{-(i_k+m_kj_k+(N+2)(v_\ell+1))}\psi_{k}(\ell,\omega_{k})< b^{-(N+2)}(1+\delta)<\delta.
$$
Furthermore, since $\widehat{\mu}$ is uniformly continuous on $\mathbb{R}$, there exists a constant $\delta_3>0$ such that
\[
\big|\widehat{\mu}(x)-\widehat{\mu}(y)\big|<\frac12\epsilon^2
\]
for any $|x-y|<\delta_1$. This uniform continuity yields that the lower bound
\[
\big|\widehat{\mu}\big(b^{-(\mathcal{I}_{k}+\mathcal{J}'_{k})}(\xi+\lambda)\big)\big|
\ge \big|\widehat{\mu}(b^{-(\mathcal{I}_{k}+\mathcal{J}'_{k})}\lambda)\big|
-\big|\widehat{\mu}\big(b^{-(\mathcal{I}_{k}+\mathcal{J}'_{k})}(\xi+\lambda)\big)-\widehat{\mu}(b^{-(\mathcal{I}_{k}+\mathcal{J}'_{k})}\lambda)\big|
\ge \frac12\epsilon^2
\]
valid for all $\xi\in[0,\delta_3]$, $\lambda\in \Lambda^{(t,s)}_{k}(\omega)$ and $k\ge 1$. Accordingly, we obtain the uniform positive lower bound
$$
\inf_{\xi\in[0,\delta_3],\,\lambda\in \Lambda^{(t,s)}_{k}(\omega)}\big|\widehat{\mu_{>\mathcal{I}_{k}+\mathcal{J}'_{k}}}(\xi+\lambda)\big|
=\inf_{\xi\in[0,\delta_3],\,\lambda\in \Lambda^{(t,s)}_{k}(\omega)}\big|\widehat{\mu}\big(b^{-(\mathcal{I}_{k}+\mathcal{J}'_{k})}(\xi+\lambda)\big)\big|
\ge \frac12\epsilon^2,
$$
It follows from Lemma \ref{lem2.1} that $\Lambda^{(t,s)}(\omega)$ is a spectrum for $\mu$.
\end{proof}

We now prove Theorem \ref{th1.3}-\ref{th1.4} using the above results.

\begin{proof}[{\bf The proof of Theorem \ref{th1.3}}] We firstly consider $t=\alpha \frac{\log q}{\log b}$ with $\alpha\in (0,1)\cap \Bbb Q^c$ and $s\in [0,\infty]$. Let $\Lambda^{(t,s)}(\omega):=\widetilde{\Lambda}^{(t,s)}(\omega)\cup \overline{\Lambda}^{(t,s)}(\omega)$ be given by \eqref{3.8}. To establish that $\dim_{Be}\big(\Lambda^{(t,s)}(\omega)\big)=t, D^+_t\big(\Lambda^{(t,s)}(\omega)\big)=s$, we first note that by Lemmas \ref{lem2.5} and \ref{lem3.1},
we have $\dim_{Be}\big(\overline{\Lambda}^{(t,s)}(\omega)\big)=0$. Therefore, by Lemma \ref{lem2.4},
it suffices to prove
\[
\dim_{Be}\big(\widetilde{\Lambda}^{(t,s)}(\omega)\big)=t,\quad D^+_t\big(\widetilde{\Lambda}^{(t,s)}(\omega)\big)=s.
\]
From Lemmas \ref{lem3.1}, we have
\[
\widetilde{\Lambda}^{(t,s)}_k(\omega)\subset \{0,1,\cdots, 2b^{\mathcal{I}_k+\mathcal{J}'_k+N}\}.
\]
It implies the cardinality estimate:
\[
\#\Big(\widetilde{\Lambda}^{(t,s)}(\omega)\cap (-2b^{\mathcal{I}_k+\mathcal{J}'_k+N},2b^{\mathcal{I}_k+\mathcal{J}'_k+N})\Big)\geq \#\widetilde{\Lambda}^{(t,s)}_k(\omega)=q^{\mathcal{J}_k}.
\]
It follows form Lemma \ref{lem2.3} and \eqref{2.2} that
\begin{align*}
\dim_{Be}\big(\widetilde{\Lambda}^{(t,s)}(\omega)\big)&= \varlimsup_{h\to\infty}\sup_{x\in\mathbb{R}}
\frac{\log \#\big(\widetilde{\Lambda}^{(t,s)}(\omega) \cap \big(x-h,\,x+h\big)\big)}{\log h}\\
&\geq \varlimsup_{k\to\infty}\frac{\log \#\big(\widetilde{\Lambda}^{(t,s)}(\omega) \cap (-2b^{\mathcal{I}_k+\mathcal{J}'_k+N},2b^{\mathcal{I}_k+\mathcal{J}'_k+N})\big)}{\log 2b^{\mathcal{I}_k+\mathcal{J}'_k}}\\
&\geq \varlimsup_{k\to\infty}\frac{\log q^{\mathcal{J}_k}}{\log 2b^{\mathcal{I}_k+\mathcal{J}'_k+N}}=t.
\end{align*}
\par From Lemma \ref{lem3.1}, when $k$ large enough, we have
\begin{align*}
\widetilde{\Lambda}^{(t,s)}_{k}(\omega)\subset \Big[0, 2b^{\mathcal{I}_k+\mathcal{J}'_k}\Big].
\end{align*}
It implies the cardinality estimate:
\[
\#\Big(\widetilde{\Lambda}^{(t,s)}(\omega)\cap \Big[0, 2b^{\mathcal{I}_k+\mathcal{J}'_k}\Big]\Big)\geq q^{\mathcal{J}_k}.
\]
It follows from \eqref{2.2} that
\begin{align*}
D^+_t\big(\widetilde{\Lambda}^{(t,s)}(\omega)\big)&= \varlimsup_{h\to\infty}\sup_{x\in\mathbb{R}}
\frac{\#\big(\widetilde{\Lambda}^{(t,s)}(\omega) \cap \big[x-h,\,x+h\big]\big)}{h^t}\geq \varlimsup_{k\to\infty}\frac{\#\big(\widetilde{\Lambda}^{(t,s)}(\omega) \cap \Big[0, 2b^{\mathcal{I}_k+\mathcal{J}'_k}\Big]}{b^{(\mathcal{I}_k+\mathcal{J}'_k)t}}\\
&\geq \varlimsup_{k\to\infty}\frac{ q^{\mathcal{J}_k}}{b^{(\mathcal{I}_{k+1}+\mathcal{J}'_{k})t}}=s.
\end{align*}
\par In the following, we prove $\dim_{Be}\big(\widetilde{\Lambda}^{(t,s)}(\omega)\big)\leq t$ and $D^+_t\big(\widetilde{\Lambda}^{(t,s)}(\omega)\big)\leq s$.
Fix $x\in\mathbb{R}$ and consider the counting set
\[
U_{x}:=\bigl(x-h,x+h\bigr)\cap\widetilde{\Lambda}^{(t,s)}(\omega).
\]
For $b^{\mathcal{I}_k+\mathcal{J}'_k}\leq h<b^{\mathcal{I}_k+\mathcal{J}'_k+1}$ with $k$ large enough, take two distinct  $\lambda_1,\lambda_2\in U_{x}$ with $\lambda_1>\lambda_2$, they can be written as
\[
\lambda_j=\sum_{r=1}^{k_{j}} b^{\mathcal{I}_{r-1}+\mathcal{J}'_{r-1}}\big(\ell_{r}^{(j)}+b^{i_{r}+m_rj_r}\phi_{i_{r}+m_rj_r}(\ell_{r}^{(j)},\omega_{r})\big)=\sum_{r=1}^{k_{j}} b^{\mathcal{I}_{r-1}+\mathcal{J}'_{r-1}}\ell_{r}^{(j)},\quad
\ell_{r}^{(j)}\in T_{i_{r}+m_rj_r}\big(\Omega_{r}\big),
\]
where $\Omega_r=\Theta^{i_r}\Xi^{(m_r-n_r)j_r}\Theta^{n_rj_r}, j=1,2$. We have following claim.
\begin{claim}\label{claim1}
We claim that $r^*:=\max\{r: \ell_{r}^{(1)}\neq \ell_{r}^{(2)}\}\leq k$.
\end{claim}
\noindent{\bf {The proof of Claim \ref{claim1} }} Otherwise, $\ell_{r^*}^{(1)}\neq \ell_{r^*}^{(2)}$ with $r^*\geq k+1$. Since $\lambda_1>\lambda_2$, then $\ell_{r^*}^{(1)}>\ell_{r^*}^{(2)}$. Indeed, if $\ell_{r^*}^{(1)}<\ell_{r^*}^{(2)}$, then
\begin{align*}
\lambda_2-\lambda_1&=b^{\mathcal{I}_{r^*-1}+\mathcal{J}'_{r^*-1}}(\ell_{r^*}^{(2)}-\ell_{r^*}^{(1)})+\sum_{r=1}^{r^*-1} b^{\mathcal{I}_{r-1}+\mathcal{J}'_{r-1}}(\ell_{r}^{(2)}-\ell_{r}^{(1)})\\
&\geq b^{\mathcal{I}_{r^*-1}+\mathcal{J}'_{r^*-1}}\cdot b^{i_{r^*}}-4 b^{\mathcal{I}_{r^*-1}+\mathcal{J}'_{r^*-1}}>0.
\end{align*}
It is a  contradiction. It follows that
\begin{align*}
\lambda_1-\lambda_2&=b^{\mathcal{I}_{r^*-1}+\mathcal{J}'_{r^*-1}}(\ell_{r^*}^{(1)}-\ell_{r^*}^{(1)})+\sum_{r=1}^{r^*-1} b^{\mathcal{I}_{r-1}+\mathcal{J}'_{r-1}}(\ell_{r}^{(1)}-\ell_{r}^{(2)})\\
&\geq b^{\mathcal{I}_{r^*-1}+\mathcal{J}'_{r^*-1}}\cdot b^{i_{r^*}}-4 b^{\mathcal{I}_{r^*-1}+\mathcal{J}'_{r^*-1}}>2b^{\mathcal{I}_k+\mathcal{J}'_k+1}>2h,
\end{align*}
which is contradict with $\lambda_1,\lambda_2\in U_{x}$. Hence we finish the claim. \qed
\par From Claim \ref{claim1}, we know that for any $\lambda\in U_{x}$, there exist $k^*$ and a common sequence $\{l^*_{r}:l^*_{r}\in T_{i_r+m_rj_r}(\Omega_r)\}_{r=k+1}^{k^*}$ such that
$$
\lambda=\sum_{r=1}^{k}  b^{\mathcal{I}_{r-1}+\mathcal{J}'_{r-1}}\ell_{r}^{(j)}+\sum_{r=k+1}^{k^*}  b^{\mathcal{I}_{r-1}+\mathcal{J}'_{r-1}}\ell_{r}^{*}.
$$
It follows that $\# U_x\leq q^{\mathcal{J}_k}$. Applying Lemma \ref{lem2.3} and \eqref{2.2} again yields
\[
\begin{aligned}
\dim_{Be}\bigl(\widetilde{\Lambda}^{(t,s)}(\omega)\bigr)
&=\varlimsup_{h\to\infty}\sup_{x\in\mathbb{R}}
\frac{\log\#\bigl(\widetilde{\Lambda}^{(t,s)}(\omega)\cap\bigl(x-h,\,x+h\bigr)\bigr)}{\log h}\\[4pt]
&\le\varlimsup_{k\to\infty}
\frac{\log q^{\mathcal{J}_{k}}}{\log b^{\mathcal{I}_k+\mathcal{J}'_k}}= t.
\end{aligned}
\]
From \eqref{2.2}, it follows that
\[
\begin{aligned}
D^+_t&=\varlimsup_{h\to\infty}\sup_{x\in\mathbb{R}}
\frac{\#\bigl(\widetilde{\Lambda}^{(t,s)}(\omega)\cap\bigl(x-h,\,x+h\bigr)\bigr)}{h^t}\\[4pt]
&\le\varlimsup_{k\to\infty}
\frac{ q^{\mathcal{J}_{k}}}{ b^{(\mathcal{I}_k+\mathcal{J}'_k)t}}= s.
\end{aligned}
\]
Hence we finish the proof of the theorem for $t=\alpha \frac{\log q}{\log b}$ with $\alpha\in (0,1)\cap \Bbb Q^c$ and $s\in [0,\infty]$.
\par The proof for $t=\alpha'\frac{\log (q^{m_0}-1)}{\log b^{m_0}}, \alpha'\in (0,1)\cap \Bbb Q^c, s\in [0,+\infty]$ is completely analogous, so we omit the details. It follows Theorem \ref{th1.3}.
\end{proof}

\begin{proof}[{\bf The proof of Theorem \ref{th1.4}}]We just prove it for $t=\alpha \frac{\log q}{\log b},\alpha\in (0,1)\cap \Bbb Q^c, s\in [0,\infty]$, since the proof for $t=\alpha'\frac{\log (q^{m_0}-1)}{\log b^{m_0}}$ with $\alpha'\in (0,1)\cap \Bbb Q^c$ is similar. From Theorem \ref{th1.3}, we know that for any $t\in (0,\frac{\log q}{\log b}), s\in [0,+\infty], \omega= \omega_1\omega_2\cdots \in \{1,2\}^\infty$, $\widetilde{\Lambda}^{(t,s)}(\omega)$ is a spectrum of $\mu$ such that $\dim_{Be}\big(\widetilde{\Lambda}^{(t,s)}(\omega)\big)=t, D^+_t\big(\widetilde{\Lambda}^{(t,s)}(\omega)\big)=s$. To show that $\mathcal{V}^{(s)}_{t}(\mu)$ has the cardinality of the continuum, we just need to prove that $\Lambda^{(t,s)}(\omega)\neq \Lambda^{(t,s)}(\eta)$ for any two distinct sequences $\omega=\omega_1\omega_2\cdots\neq \eta=\eta_1\eta_2\cdots$ in $\{1,2\}^\infty$.

\par Let $\tau=\min\{k\in\mathbb{N}^*:\omega_k\neq \eta_k\}$ denote the first discrepant position of $\omega$ and $\eta$. Without loss of generality, we assume $\omega_{\tau}=1$ and $\eta_{\tau}=2$. We first construct a distinguished element associated with the sequence $\omega$:
\[
\begin{aligned}
\lambda^{*}&=b^{\mathcal{I}_{\tau-1}+\mathcal{J}'_{\tau-1}}\bigl(b^{i_{\tau}+(m_{\tau}-n_{\tau})j_{\tau}}l_2+b^{i_{\tau}+m_{\tau}j_{\tau}}\phi_{\tau}(b^{i_{\tau}+(m_{\tau}-n_{\tau})j_{\tau}}
l_2,\omega_{\tau})\bigr)\\
&=b^{\mathcal{I}_{\tau}+\mathcal{J}'_{\tau-1}+(m_{\tau}-n_{\tau})j_{\tau}}l_2+b^{\mathcal{I}_{\tau}+\mathcal{J}'_{\tau}}l^{(b^{-(i_\tau+m_\tau j_\tau)}l_2)}_1,
\end{aligned}
\]
By the structural definition of spectral sets, every element $\lambda\in \Lambda^{(t,s)}(\eta)$ admits the following decomposition
\[
\begin{aligned}
\lambda &= \sum_{r=1}^{k} b^{\mathcal{I}_{r-1}+\mathcal{J}'_{r-1}} \psi_{r}(\ell_r,\eta_r)
+\sum_{r=c}^{k} b^{\mathcal{I}_{r}+\mathcal{J}'_{r}+(N+2)(w_r+1)}\tilde{l} \\
&=\sum_{r=1}^{k} b^{\mathcal{I}_{r-1}+\mathcal{J}'_{r-1}}(\ell_r+\phi_{r}(\ell_r,\eta_r)
+\sum_{r=c}^{k} b^{\mathcal{I}_{r}+\mathcal{J}'_{r}+(N+2)(w_r+1)}\tilde{l}
\end{aligned}
\]
for some $1\le w_r\le \sigma_r$, $1\le c\le k+1$ and $\ell_r\in \mathcal{L}_{i_r+m_rj_r}$. According to the construction of $\lambda^*$, we have the modular relation
\[
b^{-(\mathcal{I}_{\tau-1}+\mathcal{J}'_{\tau-1})}\lambda^{*}=b^{i_{\tau}+(m_{\tau}-n_{\tau})j_{\tau}} \pmod{b^{(i_{\tau}+m_{\tau}j_{\tau})}}
\]
Suppose, for contradiction, that $\lambda^{*}=\lambda$ for some $\lambda\in\Lambda^{(t,s)}(\eta)$. It is necessary that $\ell_r=0$ for all $1\le r\le \tau-1$, and the scaled element satisfies
$$
b^{-(\mathcal{I}_{\tau-1}+\mathcal{J}'_{\tau-1})}\lambda\equiv b^{i_{\tau}+(m_{\tau}-n_{\tau})j_{\tau}} \pmod{b^{i_{\tau}+m_{\tau}j_{\tau}}}.
$$
Combined with the constraint $\ell_{\tau}\in \mathcal{L}_{i_{\tau}+m_{\tau}j_{\tau}}$, this congruence forces $\ell_{\tau}=b^{i_{\tau}+(m_{\tau}-n_{\tau})j_{\tau}}$. Nevertheless, we compute the difference between $\lambda$ and $\lambda^*$:
\[
\begin{aligned}
\lambda-\lambda^{*} &= b^{\mathcal{I}_{\tau-1}+\mathcal{J}'_{\tau-1}}(\phi_{\tau}(\ell_\tau,\eta_\tau)-\phi_{\tau}(\ell_\tau,\omega_\tau))
+\sum_{r=\max\{c,\tau+1\}}^{k} b^{\mathcal{I}_{r}+\mathcal{J}'_{r}+(N+2)(w_r+1)}\tilde{l}\\
&=b^{\mathcal{I}_{\tau-1}+\mathcal{J}'_{\tau-1}}\Big(l^{(b^{-(i_\tau+m_\tau j_\tau)}l_2)}_2-l^{(b^{-(i_\tau+m_\tau j_\tau)}l_2)}_1\Big)+\sum_{r=\max\{c,\tau+1\}}^{k} b^{\mathcal{I}_{r}+\mathcal{J}_{r}+2(w_r+N)}>0,
\end{aligned}
\]
as $l^{(b^{-(i_\tau+m_\tau j_\tau)}l_2)}_2>l^{(b^{-(i_\tau+m_\tau j_\tau)}l_2)}_1$.This yields a contradiction, proving that $\lambda^{*}\notin \Lambda^{(t,s)}(\eta)$.

The above argument demonstrates that distinct symbolic sequences generate distinct spectral sets, i.e., $\Lambda^{(t,s)}(\omega)\neq \Lambda^{(t,s)}(\eta)$ whenever $\omega\neq \eta$. Since the sequence space $\{1,2\}^\infty$ has the cardinality of the continuum, the spectral family $\mathcal{V}^{(s)}_{t}(\mu)$ also possesses the cardinality of the continuum. This completes the proof.
\end{proof}

\bigskip

\noindent {\bf Conflict of interest statement}
\bigskip
\par There is no conflict of interest.

\end{document}